\documentclass[preprint]{elsarticle}

\pdfoutput=1

\makeatletter
\def\ps@pprintTitle{%
 \let\@oddhead\@empty
 \let\@evenhead\@empty
 \def\@oddfoot{\centerline{\thepage}}%
 \let\@evenfoot\@oddfoot}
\makeatother

\usepackage[utf8]{inputenc}
\usepackage[T1]{fontenc}
\usepackage{microtype}
\usepackage{xcolor}

\usepackage{mathtools, amsthm, amsfonts, amssymb}

\usepackage{mathrsfs}
\usepackage{pbox}

\usepackage{booktabs}
\usepackage{braket}
\usepackage{commath}
\usepackage{caption}

\usepackage[capitalise]{cleveref}
\usepackage{accents}

\usepackage{tikz}
\usetikzlibrary{positioning}

\theoremstyle{plain}
\newtheorem{theorem}{Theorem}
\newtheorem*{theorem*}{Theorem}
\newtheorem{proposition}[theorem]{Proposition}
\newtheorem{lemma}[theorem]{Lemma}
\newtheorem{corollary}[theorem]{Corollary}

\theoremstyle{definition}
\newtheorem{definition}[theorem]{Definition}

\newtheorem*{problem*}{Problem}

\newtheorem{remark}[theorem]{Remark}
\newtheorem{example}[theorem]{Example}

\DeclareMathOperator{\rank}{R}

\DeclareMathAccent{\wtilde}{\mathord}{largesymbols}{"65}
\DeclareMathOperator{\asymprank}{\underaccent{\wtilde}{R}}
\DeclareMathOperator{\borderrank}{\underline{R}}

\DeclareMathOperator{\tens}{T}

\newcommand{\CC}{\mathbb{C}}
\newcommand{\FF}{\mathbb{F}}

\newcommand{\NN}{\mathbb{N}}
\newcommand{\GL}{\mathrm{GL}}

\DeclareMathOperator{\type}{type}

\DeclareMathOperator{\diag}{diag}

\newcommand{\Str}{\mathrm{Str}}
\newcommand{\defin}[1]{\emph{#1}}

\newcommand{\degengeq}{\unrhd}

\def\kron{\boxtimes}
\newcommand{\tphi}{t}
\newcommand{\tpsi}{s}

\newcommand{\eps}{\varepsilon}

\DeclareMathOperator{\charac}{char}

\begin{document}

\begin{frontmatter}
\title{Tensor rank is not multiplicative under the tensor~product}
\author[cop]{Matthias Christandl}
\ead{christandl@math.ku.dk}
\author[cop]{Asger Kjærulff Jensen}
\ead{akj@math.ku.dk}
\address[cop]{Department of Mathematical Sciences, University of Copenhagen,\\ Universitetsparken 5, 2100 Copenhagen Ø, Denmark}
\author[ams]{Jeroen Zuiddam}
\ead{j.zuiddam@cwi.nl}
\address[ams]{Centrum Wiskunde \& Informatica, Science Park 123, 1098 XG Amsterdam, Netherlands}

\begin{abstract}
The tensor rank of a tensor $t$ is the smallest number $r$ such that $t$ can be decomposed as a sum of $r$ simple tensors.
Let $s$ be a $k$-tensor and let $t$ be an $\ell$-tensor. The tensor product of~$s$ and~$t$ is a $(k+\ell)$-tensor.
Tensor rank is sub-multiplicative under the tensor product. 
We revisit the connection between restrictions and degenerations. 
A result of our study is that tensor rank is not in general multiplicative under the tensor product. This answers a question of Draisma and Saptharishi.
Specifically, if a tensor $t$ has border rank strictly smaller than its rank, then the tensor rank of $t$ is not multiplicative under taking a sufficiently hight tensor product power. 
The ``tensor Kronecker product'' from algebraic complexity theory is related to our tensor product but different, namely it multiplies two $k$-tensors to get a $k$-tensor.
Nonmultiplicativity of the tensor Kronecker product has been known since the work of Strassen.

It remains an open question whether border rank and asymptotic rank are multiplicative under the tensor product. Interestingly, lower bounds on border rank obtained from generalised flattenings (including Young flattenings) multiply under the tensor product.
\end{abstract}

\end{frontmatter}

\section{Introduction}
Let $U_i, V_i$ be finite-dimensional vector spaces over a field $\FF$. Let~$\tphi$ be a $k$-tensor in $U_1 \otimes \cdots \otimes U_k$. The \defin{tensor rank} of $t$ is the smallest number~$r$ such that~$\tphi$ can be written as a sum of $r$ simple tensors $u_1 \otimes \cdots \otimes u_k$ in $U_1 \otimes \cdots \otimes U_k$, and is denoted by  $\rank(\tphi)$. Letting $\FF$ be the complex numbers $\CC$, the \defin{border rank} of~$\tphi$ is the smallest number $r$ such that $\tphi$ is a limit point (in the Euclidean topology) of a sequence of tensors in $U_1 \otimes \cdots \otimes U_k$ of rank at most $r$, and is denoted by~$\borderrank(\tphi)$.

Let $\tphi \in U_1 \otimes \cdots \otimes U_k$ and $\tpsi \in V_1 \otimes \cdots \otimes V_\ell$ be a $k$-tensor and an $\ell$-tensor respectively. Define the \defin{tensor product} of $\tphi$ and $\tpsi$ as the $(k+\ell)$-tensor
\[
\tphi \otimes \tpsi \,\in\, U_1 \otimes \cdots \otimes U_k \otimes V_1 \otimes \cdots \otimes V_\ell.
\]
If $k=\ell$, then define the \defin{tensor Kronecker  product} of $\tphi$ and $\tpsi$ as the $k$-tensor
\[
\tphi \kron \tpsi \,\in\, (U_1 \otimes V_1) \otimes \cdots \otimes (U_k \otimes V_k)
\]
obtained from $\tphi \otimes \tpsi$ by grouping $U_i$ and $V_i$ together for each $i$. In algebraic complexity theory, the tensor Kronecker product is usually just denoted by~`$\otimes$'. Using the tensor Kronecker product one defines the \defin{asymptotic rank} of $\tphi$ as the limit $\lim_{n\to\infty} \rank(\tphi^{\kron n})^{1/n}$. 
(This limit exists and equals the infimum $\inf_{n} \rank(\tphi^{\kron n})^{1/n}$, 
see for example Lemma 1.1 in \cite{strassen1988asymptotic}.) 
Asymptotic rank is denoted by~$\asymprank(\tphi)$.

This paper is about the relationship between tensor rank 
and the tensor product. It follows from the definition that rank 
is sub-multiplicative under the tensor product.
\begin{proposition}\label{multprop} Let $\tphi, \tpsi$ be any tensors. Then,
$\rank(\tphi \otimes \tpsi) \leq \rank(\tphi) \rank(\tpsi)$. 
\end{proposition}

The result of this paper is that the above inequality can be strict. 
\begin{theorem*}
Tensor rank is not in general multiplicative under tensor~product.
Specifically, if a tensor $t$ has border rank strictly smaller than its tensor rank, then the tensor rank of $t$ is not multiplicative under a taking a sufficiently high tensor power.
\end{theorem*}

The theorem answers a question posed in the lecture notes of Jan Draisma \cite[Chapter~6]{draisma} and a question of Ramprasad Saptharishi (personal communication, related to an earlier version of the survey \cite{ramprasad}). The theorem was stated as a fact in \cite[page 1097]{MR2447444}, refering to \cite{burgisser1997algebraic} for the proof; however, 
\cite{burgisser1997algebraic} studies only the tensor Kronecker product $\boxtimes$.
It has been known since the work of Strassen that tensor rank is not multiplicative under the tensor Kronecker product $\boxtimes$, see \cref{strex}.

We construct three instances of this phenomenon (\cref{main1}, \cref{main2} and \cref{mamunonmult})  to prove the theorem.  Explicitly, one of our examples is the following strict inequality (\cref{example}).
\begin{example}\label{multicounterexample}
Let $b_1, b_2$ be the standard basis of $\CC^2$. Define the 3-tensor $W_3$ as $b_2 \otimes b_1 \otimes b_1 +  b_1 \otimes b_2 \otimes b_1 +  b_1 \otimes b_1 \otimes b_2 \in (\CC^2)^{\otimes 3}$ . Then we have the strict inequality $\rank(W_3^{\otimes 2}) \leq 8 < 9 = \rank(W_3)^2$. 
\end{example}
In \cref{secpencils} we will prove that \cref{multicounterexample} is essentially minimal over the complex numbers, in the sense that if $s \in \CC \otimes \CC^2 \otimes \CC^2$ and $t\in \CC^2 \otimes \CC^n \otimes \CC^m$, then one has $\rank(s \kron t) = \rank(s \otimes t) = \rank(s) \rank(t)$. This we prove using the theory of canonical forms of matrix pencils and a formula for their tensor rank.

Our general approach is to study approximate decompositions (or border rank decompositions) of tensors. It turns out that a border rank decomposition of a tensor $t$ can be transformed into a tensor rank decomposition of tensor powers of $t$ with a penalty that depends on the so-called error degree of the approximation. More precisely, the notion of border rank $\borderrank(t)$ has a more precise variant $\rank^e(t)$ that allows only approximations with error degree at most $e$ (see \cref{degen} for definitions). This variant goes back to \cite{MR592760} and \cite{MR605920}. We prove in \cref{imp}(\ref{cor3}) that 
\begin{equation}\label{impx}
\rank(\tpsi^{\otimes n}) \leq (ne + 1) \rank^e(\tpsi)^n,
\end{equation}
which we use to construct nonmultiplicativity examples. In particular, we see that as soon as $\rank^e(s)<\rank(s)$, the quantity $\rank(s)^n$ grows faster than the right-hand side of \eqref{impx} and thus leads to nonmultiplicativity examples for large enough~$n$.

It follows from the definitions that also border rank and asymptotic rank are submultiplicative under the tensor product: $\borderrank(\tphi \otimes \tpsi) \leq \borderrank(\tphi) \borderrank(\tpsi)$, and  $\asymprank(\tphi \otimes \tpsi) \leq \asymprank(\tphi) \asymprank(\tpsi)$. We leave it as an open question whether these inequalities can be strict. 
In \cref{secyoung} we will see that lower bounds on border rank obtained from generalised flattenings (including Young flattenings) are in fact multiplicative under the tensor product.

It follows from $\rank(\tphi \kron \tpsi) \leq \rank(\tphi \otimes \tpsi)$ that tensor rank, border rank and asymptotic rank are submultiplicative under the tensor Kronecker product: $\rank(\tphi \kron \tpsi) \leq \rank(\tphi) \rank(\tpsi)$, $\borderrank(\tphi \kron \tpsi) \leq \borderrank(\tphi) \borderrank(\tpsi)$, and $\asymprank(\tphi \kron \tpsi) \leq \asymprank(\tphi) \asymprank(\tpsi)$.
If $\tphi$ and~$\tpsi$ are 2-tensors (matrices), then tensor rank, border rank and asymptotic rank are equal and multiplicative under the tensor Kronecker product.
However, for $k\geq 3$, it is well-known that each of the three inequalities 
can be strict, see the following example.

\begin{example}\label{strex}
Consider the  following tensors
\begin{align*}
\tens\Bigl(\begin{minipage}{1.1cm}
\begin{center}
\begin{tikzpicture}[vertex/.style = {circle, fill, black, minimum width = 1.mm, inner sep=0pt}]
    \path[coordinate] (0,0)  coordinate(A)
                ++( 2*1*60+30:0.8cm) coordinate(B)
                ++( 2*2*60+30:0.8cm) coordinate(C);
	\draw [line width=0.2mm]  (A) node [vertex] {} -- (B) node [vertex] {}  (C) node [vertex] {}  (A);
\end{tikzpicture}
\end{center}
\end{minipage}\Bigr) &= \sum_{i\in \{1,2\}} b_{i}\otimes b_{i}\otimes 1 \,\in\, \FF^2\otimes \FF^2 \otimes \FF, \\
\tens\Bigl(\begin{minipage}{1.1cm}
\begin{center}
\begin{tikzpicture}[vertex/.style = {circle, fill, black, minimum width = 1.mm, inner sep=0pt}]
    \path[coordinate] (0,0)  coordinate(A)
                ++( 2*1*60+30:0.8cm) coordinate(B)
                ++( 2*2*60+30:0.8cm) coordinate(C);
	\draw [line width=0.2mm]  (A) node [vertex] {}  (B) node [vertex] {} --  (C) node [vertex] {}  (A);
\end{tikzpicture}
\end{center}
\end{minipage}\Bigr)&= \sum_{i\in \{1,2\}} b_{i}\otimes1\otimes b_{i} \,\in\, \FF^2 \otimes \FF \otimes \FF^2, \\
\tens\Bigl(\begin{minipage}{1.1cm}
\begin{center}
\begin{tikzpicture}[vertex/.style = {circle, fill, black, minimum width = 1.mm, inner sep=0pt}]
    \path[coordinate] (0,0)  coordinate(A)
                ++( 2*1*60+30:0.8cm) coordinate(B)
                ++( 2*2*60+30:0.8cm) coordinate(C);
	\draw [line width=0.2mm]  (A) node [vertex] {}  (B) node [vertex] {}  (C) node [vertex] {} --  (A);
\end{tikzpicture}
\end{center}
\end{minipage}\Bigr)&= \sum_{i\in \{1,2\}} 1\otimes b_{i}\otimes b_{i} \,\in\, \FF \otimes \FF^2 \otimes \FF^2.
\end{align*}
(This graphical notation is borrowed from \cite{christandl2016asymptotic}.)
Each tensor has rank, border rank and asymptotic rank equal to~2, since they are essentially identity matrices. However the tensor Kronecker product is the $2\times 2$ matrix multiplication tensor
\[
\langle 2,2,2 \rangle = \tens\Bigl(\begin{minipage}{1.1cm}
\begin{center}
\begin{tikzpicture}[vertex/.style = {circle, fill, black, minimum width = 1.mm, inner sep=0pt}]
    \path[coordinate] (0,0)  coordinate(A)
                ++( 2*1*60+30:0.8cm) coordinate(B)
                ++( 2*2*60+30:0.8cm) coordinate(C);
	\draw [line width=0.2mm]  (A) node [vertex] {} -- (B) node [vertex] {}--  (C) node [vertex] {}--  (A);
\end{tikzpicture}
\end{center}
\end{minipage}\Bigr) = \sum_{i,j,k\in \{1,2\}}(b_{i}\otimes b_j)\otimes(b_j\otimes b_k)\otimes (b_k \otimes b_i)
\]
whose tensor rank and border rank is at most 7 \cite{strassen1969gaussian} and whose asymptotic rank is thus at most 7, which is strictly less that $2^3=8$. (The tensor rank of~$\langle 2,2,2\rangle$ equals 7 over any field \cite{winograd1971multiplication} and the border rank of $\langle 2,2,2 \rangle$ equals 7 over the complex numbers $\CC$ \cite{landsberg2006border}. Both statements are in fact true for any tensor with the same support as $\langle 2,2,2 \rangle$ \cite{blaser2017border}.)
\end{example}

\section{Degeneration and restriction}\label{degen}

We revisit the theory of degenerations and restrictions of tensors and how to transform degenerations into restrictions. Our non-multiplicativity results rely on these ideas. Let $\tphi\in U_1\otimes \cdots \otimes U_k$ and $\tpsi \in V_1 \otimes \cdots \otimes V_k$ be $k$-tensors.
We say $\tphi$ \defin{restricts} to $\tpsi$, written $\tphi \geq \tpsi$, if there are linear maps $A_i : U_i \to V_i$ such that $(A_1 \otimes \cdots \otimes A_k) \tphi = \tpsi$. Let $d,e\in \NN$. We say $\tphi$ \defin{degenerates to}~$\tpsi$ with approximation degree $d$ and error degree $e$, written $\tphi\degengeq_d^e\tpsi$, if there are linear maps $A_i(\varepsilon) : U_i \to V_i$ depending polynomially on~$\varepsilon$ such that $(A_1(\varepsilon) \otimes \cdots \otimes A_k(\varepsilon)) \tphi = \varepsilon^d \tpsi + \varepsilon^{d+1} \tpsi_1 + \cdots + \varepsilon^{d+e} \tpsi_e$ for some tensors $\tpsi_1, \ldots, \tpsi_e$. 
Naturally, $\tphi\degengeq^e\tpsi$ means $\exists d\colon \tphi\degengeq^e_d \tpsi$, and $\tphi\degengeq_d\tpsi$ means $\exists e\colon \tphi\degengeq^e_d \tpsi$, and $\tphi \degengeq \tpsi$ means $\exists d\,\exists e\colon \tphi\degengeq^e_d \tpsi$.
(We note that our notation $t\degengeq_d s$ corresponds to $t\degengeq_{d+1}s$ in \cite{burgisser1997algebraic}.)
Clearly, degeneration is multiplicative in the following sense.

\begin{proposition}\label{mult}
Let $\tphi_1,\tphi_2,\tpsi_1,\tpsi_2$ be tensors.
If $\tphi_1 \degengeq^{e_1}_{d_1} \tpsi_1$ and $\tphi_2 \degengeq^{e_2}_{d_2} \tpsi_2$, then $\tphi_1 \otimes \tphi_2 \degengeq^{e_1 + e_2}_{d_1 + d_2} \tpsi_1 \otimes \tpsi_2$ and $\tphi_1 \kron \tphi_2 \degengeq^{e_1 + e_2}_{d_1 + d_2} \tpsi_1 \kron \tpsi_2$.
\end{proposition}

The error degree $e$ is upper bounded by the approximation degree $d$ in the following way.

\begin{proposition}\label{errordeg} Let $\tphi, \tpsi$ be $k$-tensors.
If $\tphi \degengeq_d \tpsi$, then $\tphi \degengeq_d^{kd - d} \tpsi$.
\end{proposition}
\begin{proof}
Suppose $(A_1(\eps) \otimes \cdots \otimes A_k(\eps)) t = \eps^d s + \eps^{d+1} s_1 + \cdots + \eps^{d+e} s_e$. For every~$i$ let $B_i(\eps)$ be the matrix obtained from $A_i(\eps)$ by truncating each entry in $A_i(\eps)$ to degree at most $d$. Then $(B_1(\eps) \otimes \cdots \otimes B_k(\eps)) t = \eps^d s + \eps^{d+1} u_1 + \cdots + \eps^{kd} u_{kd}$ for some $k$-tensors $u_1, \ldots, u_{kd}$.
\end{proof}

For any $r\in \NN$, let $b_1, \ldots b_r$ denote the standard basis of~$\FF^r$. Let $r, k\in \NN$ and~let
\[
\tens_r(k) \coloneqq \sum_{i=1}^r (b_i)^{\otimes k} \,\in\, (\FF^r)^{\otimes k}
\]
be the rank-$r$ order-$k$ \defin{unit tensor}. Let $\tpsi \in V_1 \otimes \cdots \otimes V_k$. The tensor rank of~$\tpsi$ is the smallest number~$r$ such that $\tens_r(k)\geq \tpsi$, and is denoted by $\rank(\tpsi)$. This definition of tensor rank is easily seen to be equivalent to the definition given in the introduction.
The border rank of $\tpsi$ is the smallest number $r$ such that $\tens_r(k) \degengeq \tpsi$, and is denoted by $\borderrank(\tpsi)$. Note that this definition works over any field~$\FF$. When $\FF$ equals $\CC$, this definition of border rank is equivalent to the definition given in the introduction \cite{alder1984grenzrang, strassen1987relative, lehmkuhl1989order, burgisser1997algebraic}. Define
\begin{align*}
\rank^e_d(\tpsi) &\coloneqq \min\{r \in \NN \mid \tens_r(k) \degengeq_d^e \tpsi\}\\
\rank_d(\tpsi) &\coloneqq \min\{r \in \NN \mid \tens_r(k) \degengeq_d \tpsi\}\\
\rank^e(\tpsi) &\coloneqq \min\{r \in \NN \mid \tens_r(k) \degengeq^e \tpsi\}.
\end{align*}
(Our notation $\rank_d(\tpsi)$ corresponds to $\rank_{d+1}(\tpsi)$ in \cite{burgisser1997algebraic}.) Error degree in the context of border rank was already studied in \cite{MR592760} and \cite{MR605920}.
The following propositions follow directly from \cref{mult} and \cref{errordeg}.

\begin{proposition}\label{degreemult}
$\rank^{e_1 + e_2}_{d_1 + d_2}(\tpsi_1 \otimes \tpsi_2) \leq \rank^{e_1}_{d_1}(\tpsi_1) \rank^{e_2}_{d_2}(\tpsi_2)$.
\end{proposition}

\begin{proposition}\label{eub} Let $\tpsi$ be a $k$-tensor. Then
$\rank_d(\tpsi) = \rank^{kd - d}_d(\tpsi)$.
\end{proposition}

The following theorem is our main technical result on which the rest of the paper rests.
We note that for the tensor Kronecker product the statement is well-known in the context of algebraic complexity theory \cite{MR592760, MR605920, schonhage1981partial, strassen1991degeneration, christandl2016asymptotic}.

\begin{theorem}\label{maine}
Let $\tphi, \tpsi$ be $k$-tensors.
If $\tphi \degengeq^e \tpsi$ and $|\FF|\geq e+2$, then we have $t \kron \tens_{e+1}(k) \geq \tpsi$.
\end{theorem}
\begin{proof}
By assumption there are matrices $A_i(\varepsilon)$ with entries polynomial in $\varepsilon$ such that
\begin{equation*}
\bigl(A_1(\varepsilon) \otimes \cdots \otimes A_k(\varepsilon)\bigr) \tphi = \varepsilon^d \tpsi + \varepsilon^{d+1} \tpsi_1 + \cdots + \varepsilon^{d+e} \tpsi_e
\end{equation*}
for some tensors $\tpsi_1, \ldots, \tpsi_e$. 
%
Multiply both sides 
by $\varepsilon^{-d}$ and call the right-hand side $q(\varepsilon)$,
\[
\bigl( \varepsilon^{-d} A_1(\varepsilon) \otimes \cdots \otimes A_k(\varepsilon)\bigr) \tphi =  \tpsi + \varepsilon \tpsi_1 + \cdots + \varepsilon^{e} \tpsi_e \eqqcolon q(\varepsilon).
\]
Let $\alpha_0, \ldots, \alpha_e$ be distinct nonzero elements of the ground field $\FF$ (by assumption our ground field is large enough to do this). View $q(\varepsilon)$ as a polynomial in $\varepsilon$.  Write $q(\varepsilon)$ as follows (Lagrange interpolation):
\[
q(\varepsilon) = \sum_{j=0}^e q(\alpha_j) \prod_{\substack{0 \leq m \leq e:\\ m\neq j}} \frac{\varepsilon - \alpha_m}{\alpha_j - \alpha_m}.
\]
We now see how to write $q(0)$ as a linear combination of the $q(\alpha_j)$, namely
\[
q(0) = \sum_{j=0}^e q(\alpha_j) \prod_{\substack{0 \leq m \leq e:\\ m\neq j}} \frac{\alpha_m}{\alpha_m - \alpha_j},
\]
that is,
\[
q(0) = \sum_{j=0}^e  \beta_j\, q(\alpha_j) \quad \textnormal{with} \quad \beta_j \coloneqq \prod_{\substack{0 \leq m \leq e:\\ m\neq j}} \frac{\alpha_m}{\alpha_m - \alpha_j}.
\]
Now we want to write $\tpsi$ as a restriction of $\tphi \kron \tens_{e+1}(k)$. Define the  linear maps $B_1 \coloneqq \sum_{j=0}^{\smash{e}} \beta_j\,\,\alpha_j^{-d}\, A_1(\alpha_j) \otimes b_j^*$ and
$B_i \coloneqq \sum_{\smash{j=0}}^{e} \beta_j\,\, A_i(\alpha_j) \otimes b_j^*$ for $i \in \{2,\ldots, k\}$.
Then $t\kron \tens_{e+1}(k) \geq s$ because
\begin{align*}
(B_1 \otimes \cdots \otimes B_k) (\tphi\kron \tens_{e+1}(k)) &= \sum_{j=0}^e \beta_j\, \bigr(\alpha_j^{-d}A_1(\alpha_j) \otimes \cdots \otimes A_k(\alpha_j)\bigr) \tphi\\
&= \sum_{j=0}^e \beta_j\, q(\alpha_j) = q(0) = \tpsi.
\end{align*}
This finishes the proof.
\end{proof}

\begin{remark}
In the statement of \cref{maine} we assume that $\abs[0]{\FF}$ is large enough. For small fields one can do the following.
For $k,d\in \NN$, let $[0..d]$ denote the set $\{0,1,2,\ldots, d\}$ and define the $k$-tensor
\[
\chi_d(k) \coloneqq \sum_{\mathclap{\substack{a \in [0..d]^k:\\a_1 + \cdots + a_k = d}}} b_{a_1} \otimes \cdots \otimes b_{a_k} \in (\FF^{d+1})^{\otimes k}.
\]
Let $\tphi, \tpsi$ be $k$-tensors. It is not hard to show that, if $\tphi \degengeq_d \tpsi$, then $t\kron \chi_d(k) \geq s$.
By definition of $\chi_d(k)$ we have $\rank(\chi_d(k)) \leq \binom{k+d-1}{k-1}$.
%
We may thus conclude that 
$t \boxtimes \tens_{\binom{k+d-1}{k-1}}(k) \geq s$. 
\end{remark}

We collect several almost immediate corollaries.

\begin{corollary}\label{corx}
 Let $\tphi_i, \tpsi_i$ be $k_i$-tensors for $i\in [n]$. Assume $\FF$ is large enough.
\begin{enumerate} 
\item\label{cor1} If $\forall i\colon \tphi_i \degengeq^{e_i} \tpsi_i$, then $(\tphi_1 \otimes \cdots \otimes \tphi_n) \kron \tens_{\sum_i e_i + 1}(\sum_i k_i) \geq \tpsi_1 \otimes \cdots \otimes \tpsi_n$.
\item\label{errorver} If $\forall i\colon \tphi_i \degengeq_{d_i} \tpsi_i$, then $(\tphi_1 \otimes \cdots \otimes \tphi_n) \kron \tens_{\sum_i (k_i -1)d_i + 1}(\sum_i k_i) \geq \tpsi_1 \otimes \cdots \otimes \tpsi_n$.
\end{enumerate}
\end{corollary}
\begin{proof}
To prove the first statement, apply \cref{mult} to obtain the degeneration $\tphi_1 \otimes \cdots \otimes \tphi_n \degengeq^{\sum_i e_i} \tpsi_1 \otimes \cdots \otimes \tpsi_n$.
\cref{maine} yields the result.
To prove the second statement, 
 \cref{errordeg} gives $t_i \degengeq^{k_i d_i - d_i} s_i$. By \cref{mult}, $t_1 \otimes \cdots \otimes t_n \degengeq^{\sum_i k_i d_i - d_i } s_1 \otimes \cdots \otimes s_n$. \cref{maine} proves the statement.
\end{proof}

\begin{corollary}\label{imp} Let $\tpsi$ be a $k$-tensor. Assume $\FF$ is large enough.
\begin{enumerate}
\item\label{cor3}\label{cor2} $\rank(\tpsi^{\otimes n}) \leq (ne + 1) \rank^e(\tpsi)^n$.
\item\label{maind} $\rank(\tpsi^{\otimes n}) \leq ((k - 1)nd + 1) \rank_{d}(\tpsi)^n$.
\end{enumerate}
\end{corollary}
\begin{proof}
This follows from \cref{corx}.
\end{proof}

\begin{corollary}  Let $\tpsi$ be a $k$-tensor.
\begin{enumerate}
\item $\lim_{n\to \infty} \rank(\tpsi^{\otimes n})^{1/n} \leq \borderrank(\tpsi)$.
\item $\lim_{n\to \infty} \rank(\tpsi^{\otimes n})^{1/n} = \lim_{n\to\infty} \borderrank(\tpsi^{\otimes n})^{1/n}$.
\item If $\borderrank(\tpsi) < \rank(\tpsi)$, then for some $n \in \NN$, $\rank(\tpsi^{\otimes n}) < \rank(\tpsi)^n$.
\end{enumerate}
\end{corollary}


\section{Tensor rank is not multiplicative under the tensor product}\label{sec3}

Because of \cref{imp}, in order to find nonmultiplicativity examples, it is enough to find a tensor $t$ for which $\rank^e(t) < \rank(t)$. We will give three families of examples of nonmultiplicativity. For $k\geq 3$, define the $k$-tensor
\[
W_k \coloneqq  \sum_{\mathclap{\substack{i \in \{1,2\}^k:\\ \type(i) = (k-1,1)}}} b_{i_1} \otimes \cdots \otimes b_{i_k} \,\in\, (\FF^2)^{\otimes k},
\]
where $\type(i) = (k-1,1)$ means that $i$ is a permutation of $(1,1,\ldots, 1, 2)$.

\begin{proposition}\label{main1} Let $|\FF|$ be large enough. Let $k\geq 3$. For $n$ large enough, we have a strict inequality
$\rank(W_k^{\otimes n}) < \rank(W_k)^n$. For example, $\rank(W_3^{\otimes 7}) < \rank(W_3)^7$ and $\rank(W_8^{\otimes 2}) < \rank(W_8)^2$.
\end{proposition}
\begin{proof}
The rank of~$W_k$ equals $k$. This can be shown with the substitution method as explained in for example~\cite{blaser2013fast}. However, $\rank^{k-1}(W_k) \leq 2$, namely
\[
\bigl( \begin{psmallmatrix} 1 & 1\\ \varepsilon & 0 \end{psmallmatrix} \otimes \cdots \otimes \begin{psmallmatrix} 1 & 1\\ \varepsilon & 0 \end{psmallmatrix} \otimes \begin{psmallmatrix} 1 & -1\\ \varepsilon & 0 \end{psmallmatrix} \bigr) \tens_2(k) = \varepsilon W_k + \varepsilon^2 ({\cdots}) + \cdots + \varepsilon^k\, (b_2 \otimes \cdots \otimes b_2).
\]
Applying \cref{imp}(\ref{cor3}) to this degeneration gives $\rank(W_k^{\otimes n}) \leq (n(k-1) + 1)2^n$.
Therefore, for $n$ large enough, $\rank(W_k^{\otimes n}) \leq 2^n(n(k-1)+1) < k^n = \rank(W_k)^n$. 
\end{proof}

In fact, if $\charac(\FF) \neq 2$ and $\sqrt{2} \in \FF$, then we can directly show a strict inequality for $n=2$ and $k=3$ as follows.  

\begin{proposition}\label{example}
$\rank(W_3^{\otimes 2}) \leq 8 < 9 = \rank(W_3)^2$ if $\charac{\FF} \neq 2$ and $\sqrt{2} \in \FF$.
\end{proposition}
\begin{proof}
As mentioned in the proof of \cref{main1}, $\rank(W_3) = 3$.
If $c\in \FF\setminus\{0\}$ such that $\sqrt{c} \in \FF$, then $\rank(W_3 + c\, b_2 \otimes b_2 \otimes b_2) \leq 2$. Namely,
\[
W_3 + c\, b_2 \otimes b_2 \otimes b_2 = \frac{1}{2\sqrt{c}}\bigl( (b_1 + \sqrt{c}\, b_2)^{\otimes 3} - (b_1 - \sqrt{c}\, b_2)^{\otimes 3}\bigr).
\]
(Over $\CC$ this also follows from the fact that the \emph{Cayley hyperdeterminant} evaluated at $W_3 + c\, b_2 \otimes b_2 \otimes b_2$ is a nonzero constant times~$c$. One may also see this by noting that the image of $W_3 + c\, b_2 \otimes b_2 \otimes b_2$ under the \emph{moment map} lies outside the image of the \emph{moment polytope} associated to the orbit  $\GL_2 \times \GL_2 \times \GL_2 \cdot W$ \cite{MR3087706,MR3195184}.)
We expand $W_3\otimes W_3$ as
\begin{align*}
W_3 \otimes W_3 = \bigl(W_3 + b_2 \otimes b_2 \otimes b_2\bigr)^{\otimes 2} &- \bigl(W_3 + \tfrac12 b_2 \otimes b_2 \otimes b_2\bigr)\otimes  b_2 \otimes b_2 \otimes b_2\\ &- b_2 \otimes b_2 \otimes b_2 \otimes \bigl(W_3 + \tfrac12 b_2 \otimes b_2 \otimes b_2\bigr).
\end{align*}
By the above, we know that the rank of $W_3 + b_2 \otimes b_2 \otimes b_2$ and the rank of $W_3 + \tfrac12 b_2 \otimes b_2 \otimes b_2$ are at most 2. Therefore, the rank of $W_3\otimes W_3$ is at most $2^2 + 2 + 2 = 8$.
\end{proof}

\begin{remark}
Let $S_k$ be the symmetric group of order $k$.
Clearly the tensor $W_3 \otimes W_3$ is invariant under the action of the subgroup $S_3 \times S_3 \subseteq S_6$ and under the action of the permutation $(14)(25)(36) \in S_6$ that swaps the two copies of $W_3$. Remarkably, the decomposition of $W_3 \otimes W_3$ given in the proof of \cref{example} also has this symmetry, in the sense that the above actions leave the set of simple terms appearing in the decomposition invariant.
The decomposition is said to be \emph{partially symmetric}.
In fact, each term is itself invariant under $S_3 \times S_3$.
\end{remark}

\begin{remark}
It is stated in \cite{yu2010tensor} that $\rank(W_3\kron W_3) = 7$, which implies that $\rank(W_3\otimes W_3)$ equals 7 or 8. We obtained numerical evidence pointing to 8. After the first version of our manuscript appeared on the arXiv, Chen and Friedland delivered a proof that $\rank(W_3\otimes W_3) \geq 8$ \cite{chen2017tensor}.  For the third power, it is known that $\rank(W_3 \kron W_3 \kron W_3) = 16$ \cite{MR3632234}. A similar construction as in the proof of \cref{example} gives $\rank(W_3\otimes W_3 \otimes W_3) \leq 21$.  This upper bound is improved to~20 in \cite{chen2017tensor}.
\end{remark}

In \cref{main1}, we took the $n$th power of a tensor in $(\FF^2)^{\otimes k}$ with $n$ large enough depending on $k$.
In our next example, we take the square of a tensor in $(\FF^d)^{\otimes k}$ with $d\geq 8$. 
For $k\geq 3$ and $q\geq 1$, define the tensor
\[
\Str_q^k \coloneqq \sum_{i=2}^{q+1} b_i \otimes b_i \otimes b_1 \otimes b_1^{\otimes k-3} + b_1 \otimes b_i \otimes b_i \otimes b_1^{\otimes k-3}\,\in\,(\FF^{q+1})^{\otimes k}.
\]
This tensor is named after Strassen, who used $\Str_q^3$ to derive the upper bound $\omega \leq 2.48$ on the exponent of matrix multiplication \cite{strassen1987relative, blaser2013fast}.

\begin{proposition}\label{main2} Assume that $\FF$ is large enough. For $q \geq 7$ and any $k\geq 3$, we have a strict inequality
$\rank((\Str_q^k)^{\otimes 2}) < \rank(\Str_q^k)^2$.
\end{proposition}
\begin{proof}
The rank of $\Str_q^k$ equals $2q$, again by the substitution method. We have $\rank^1(\Str_q^k)\leq q+1 $, see the proof of Proposition~31 in~\cite{buhrman2016nondeterministic}. Applying \cref{imp}(\ref{cor3}) to this degeneration gives $\rank((\Str_q^k)^{\otimes n}) \leq (n+1)(q+1)^n$. Therefore, for $q\geq 7$ and $n = 2$, we have the strict inequality $\rank((\Str_q^k)^{\otimes 2}) \leq 3(q+1)^2 < (2q)^2 = \rank(\Str_q^k)^2$.
\end{proof}

Our third example uses matrix multiplication tensors. Let~$n_1, n_2, n_3\in \NN$. Define the 3-tensor
\begin{multline*}
\langle n_1,n_2,n_3 \rangle \coloneqq \sum_{\mathclap{\substack{i \in [n_1]\times [n_2]\times [n_3]}}} (b_{i_1}\otimes b_{i_2})\otimes (b_{i_2}\otimes b_{i_3}) \otimes (b_{i_3} \otimes b_{i_1})\\[-0.5em]
\in (\FF^{n_1}\otimes \FF^{n_2}) \otimes (\FF^{n_2}\otimes \FF^{n_3}) \otimes (\FF^{n_3} \otimes \FF^{n_1}).
\end{multline*}

\begin{proposition}\label{mamunonmult}
 Assume that $\FF$ is large enough. For $n\geq78$, we have a strict inequality $\rank(\langle 2,2,4 \rangle^{\otimes n}) < \rank(\langle 2,2,4 \rangle)^n$.
\end{proposition}
\begin{proof}
The rank of $\langle 2,2,4\rangle$ equals 14 over any field \cite[Theorem~2]{MR0274293}. On the other hand, $\rank^4(\langle 2,2,4\rangle) \leq 13$ over any field \cite[Theorem~1]{MR3457248}. 
Thus, when $\FF$ is large enough \cref{imp}(\ref{cor3}) implies, for $n\geq 78$, the strict inequality $\rank(\langle 2,2,4 \rangle^{\otimes n}) \leq 13^n (4n+1) < 14^n = \rank(\langle 2,2,4\rangle)^n$.
\end{proof}

In the language of graph tensors \cite{christandl2016asymptotic}, \cref{mamunonmult} says that tensor rank is not multiplicative under taking disjoint unions of graphs.

\section{Generalised flattenings are multiplicative}\label{secyoung}

In the previous section we have seen that tensor rank can be strictly submultiplicative under the tensor product. We do not know whether the same is true for border rank. In fact, in this section we observe that lower bounds on border rank obtained from generalised flattenings are multiplicative. In this section we focus on 3-tensors for notational convenience. The ideas directly extend to $k$-tensors for any $k$.

Let $t$ be a tensor in $V_1 \otimes V_2 \otimes V_3$. We can transform $t$ into a matrix by grouping the tensor legs into two groups
\begin{align*}
V_1 \otimes V_2 \otimes V_3 &\to V_1 \otimes (V_2 \otimes V_3)\\[0.1em]
v_1 \otimes v_2 \otimes v_3 &\mapsto v_1 \otimes (v_2 \otimes v_3).
\end{align*}
(There are three ways to do this for a 3-tensor.)
This is called \emph{flattening}. The rank of a flattening of $t$ is a lower bound for the border rank of $t$. (Rank and border rank are equal for matrices.) 

We now define generalised flattenings. Let $t$ be a tensor in $V_1 \otimes V_2 \otimes V_3$. Instead of a basic flattening $V_1 \otimes V_2 \otimes V_3 \to V_1 \otimes (V_2 \otimes V_3)$, we choose vector spaces $V'_1$ and $V'_2$ and apply some linear map $F: V_1 \otimes V_2 \otimes V_3 \to V'_1 \otimes V'_2$ to~$t$. To obtain a border rank lower bound using $F$ we have to compensate for the fact that $F$ possibly increases the border rank of a simple tensor. The following lemma describes the resulting lower bound.

\begin{lemma}\label{simpleflattening} Let $t\in V_1 \otimes V_2 \otimes V_3$ be a tensor.
Let
\[
F: V_1 \otimes V_2 \otimes V_3 \to V'_1 \otimes V'_2
\]
be a linear map. The border rank of $t$ is at least
\begin{equation}\label{eqc}
\borderrank(t) \geq \frac{\rank(F(t))}{\max \rank(F(v_1 \otimes v_2 \otimes v_3))},
\end{equation}
where the maximum is over all simple tensors $v_1\otimes v_2 \otimes v_3$ in $V_1\otimes V_2 \otimes V_3$.
\end{lemma}
\begin{proof}
Suppose $\borderrank(t) = r$. Then there is a sequence of tensors $t_i$ converging to $t$ with $\rank(t_i) \leq r$ for each $i$. Each $t_i$ thus has a decomposition into simple tensors $t_i = \sum_{j=1}^r t_{i,j}$. Since $F(t_i) \to F(t)$, there exists an $i_0$ such that for all $i\geq i_0$ we have $\rank(F(t_i)) \geq \rank(F(t))$. Moreover, we have the inequalities $\rank(F(t_i)) \leq \sum_{j=1}^r \rank(F(t_{i,j})) \leq r \cdot \max_{s} \rank(F(s))$, where the maximum is over all simple tensors $s$. We conclude that $\borderrank(t) \geq \rank(F(t))/\max_{s} \rank(F(s))$.
\end{proof}

Note that the right hand side of \eqref{eqc} might not be an integer.
The lower bound in \eqref{eqc} is multiplicative under the tensor product in the following sense.

\begin{proposition}\label{youngmult}
Let $\tpsi \in V_1 \otimes V_2 \otimes V_3$ and $\tphi \in W_1 \otimes W_2 \otimes W_3$ be tensors. Let $F_1: V_1 \otimes V_2 \otimes V_3 \to V'_1 \otimes V'_2$ and $F_2: W_1 \otimes W_2 \otimes W_3 \to W'_1 \otimes W'_2$ be linear maps.  The border rank of $\tpsi \otimes \tphi \in V_1 \otimes V_2 \otimes V_3 \otimes W_1 \otimes W_2 \otimes W_3$ is at least
\[
\borderrank(\tpsi \otimes \tphi) \geq \frac{\rank(F_1(\tpsi))}{\max \rank(F_1(v_1 \otimes v_2 \otimes v_3))} \frac{\rank(F_2(\tphi))}{ \max \rank(F_2(w_1 \otimes w_2 \otimes w_3)) }
\]
where the maximisations are over simple tensors in $V_1 \otimes V_2 \otimes V_3$ and in $W_1 \otimes W_2 \otimes W_3$ respectively.
\end{proposition}
\begin{proof}
Combine $F_{1}$ and $F_{2}$ into a single linear map 
\[
F : V_1 \otimes V_2 \otimes V_3 \otimes W_1 \otimes W_2 \otimes W_3 \to (V'_1 \otimes W'_1) \otimes (V'_2 \otimes W'_2).
\]
One then follows the proof of \cref{simpleflattening} and uses the fact that matrix rank is multiplicative under the tensor Kronecker product.
\end{proof}

Young flattenings \cite{strassen1983rank,landsberg2011new} are a special case of generalised flattenings.
For completeness, we finish with a concise description of Young flattenings and the corresponding multiplicativity statement. We work over the complex numbers $\CC$.
Let $S_\lambda V$ be an irreducible $\GL_V$-module of type~$\lambda$. Consider the space $V \otimes S_\lambda V$ as a $\GL_V$-module under the diagonal action. The \emph{Pieri rule} says that we have a $\GL_V$-decomposition
\[
V \otimes S_\lambda V \cong \bigoplus_\mu S_\mu V,
\]
where the direct sum is over partitions $\mu$ of length at most $\dim V$ obtained from $\lambda$ by adding a box in the Young diagram of $\lambda$. This decomposition yields $\GL_V$-equivariant embeddings $S_\mu V \hookrightarrow V\otimes S_\lambda V$\!, called \emph{Pieri inclusions} or \emph{partial polarization maps}. These maps are unique up to scaling. Such a Pieri inclusion corresponds to a $\GL_V$-equivariant map $\phi_{\mu,\lambda}:V^* \to S_\mu V^* \otimes S_\lambda V$. Every element $\phi_{\mu, \lambda}(v)$ is called a \emph{Pieri map}. The~\emph{Young flattening} $F_{\mu, \lambda}$ on $V_1 \otimes V_2^* \otimes V_3$ is obtained by first applying the map~$\phi_{\mu, \lambda}$ to one tensor leg,
\[
V_1 \otimes V_2^* \otimes V_3 \to V_1 \otimes S_\mu V_2^* \otimes S_\lambda V_2 \otimes V_3,
\]
and then flattening into a matrix,
\[
V_1 \otimes S_\mu V_2^* \otimes S_\lambda V_2 \otimes V_3 \to (V_1 \otimes S_\mu V_2^*) \otimes (S_\lambda V_2 \otimes V_3).
\]
Note that for any simple tensor $v_1\otimes v_2 \otimes v_3$, the rank of $F_{\mu,\lambda}(v_1 \otimes v_2 \otimes v_3)$ equals the rank of $\phi_{\mu, \lambda}(v_2)$. \cref{youngmult} thus specialises as follows.

\begin{proposition}
Let $\tpsi \in V_1 \otimes V_2 \otimes V_3$ and $\tphi \in W_1 \otimes W_2 \otimes W_3$. 
Let $\lambda, \mu$ and $\nu, \kappa$ be pairs of partitions as above.  The border rank of $\tpsi \otimes \tphi \in V_1 \otimes V_2 \otimes V_3 \otimes W_1 \otimes W_2 \otimes W_3$ is at least
\[
\borderrank(\tpsi \otimes \tphi) \geq \frac{\rank(F_{\mu, \lambda}(\tpsi))}{\max \rank(\phi_{\mu, \lambda}(v_2))} \frac{\rank(F_{\nu, \kappa}(\tphi))}{ \max \rank(\phi_{\nu, \kappa}(w_2)) }
\]
where the maximisations are over $v_2 \in  V_2$ and $w_2 \in W_2$ respectively.
\end{proposition}

We refer to \cite{landsberg2012tensors} for an overview of the applications of Young flattenings.

\section{Multiplicativity for complex matrix pencils and 2-tensors}\label{secpencils}
In this section all vector spaces are over the complex numbers. 
The goal of this section is to prove the following proposition.
\begin{proposition}\label{pencilprop}
Let $\tpsi \in \CC \otimes \CC^d \otimes \CC^d $ and $\tphi\in \CC^2\otimes \CC^n\otimes \CC^m$. Then
\[ 
R(\tphi\kron\tpsi) = R(\tphi\otimes\tpsi) =R(\tphi)R(\tpsi).
\]
\end{proposition}


\begin{remark}
\cref{pencilprop} shows that \cref{multicounterexample} is essentially minimal over the complex numbers. 
Namely, any example of non-multiplicativity of tensor rank under $\otimes$ must either be with a 5-tensor in $(\CC^{d}\otimes\CC^{d})\otimes(\CC^{d_1}\otimes \CC^{d_2}\otimes\CC^{d_3})$ with $d_1,d_2,d_3\ge 3$, $d\ge 2$ or in a tensor space of order 6 or more. 
Moreover, one can show using \cref{pencilprop} and the well-known classification of the $\GL_2^{\times 3}$-orbits in $\CC^2 \otimes \CC^2 \otimes \CC^2$ that if $s,t \in \CC^2 \otimes \CC^2 \otimes \CC^2$ and $\rank(s \otimes t) < \rank(s)\rank(t)$, then $s$ and~$t$ are both isomorphic to the tensor $W_3$.
\end{remark}

The elements of $\CC^2\otimes \CC^n\otimes \CC^m$ are often called matrix pencils. The tensor rank of matrix pencils is completely understood, in the sense that every matrix pencil is equivalent under local isomorphisms to a pencil in canonical form, 
for which the rank is given by a simple formula. 
This formula will allow us to give a short proof of \cref{pencilprop}.

We begin with introducing the canonical form for matrix pencils. 
For a proof 
we refer to \cite[Chapter~XII]{zbMATH01235882}.  Recall that the standard basis elements of $\CC^n$ are denoted by $b_1, \ldots, b_n$.

\begin{definition}
Given $t_i\in U \otimes V_i \otimes W_i$, define $\diag_{U}(t_1,\ldots,t_n)$
as the image of $\bigoplus_{i=1}^n t_i$ under the natural inclusion $\bigoplus_i(U\otimes V_i\otimes W_i)\to U\otimes\bigl(\bigoplus_i V_i \bigr)\otimes\bigl(\bigoplus_i W_i\bigr)$.
For $\eps\in \NN$ define the tensor $L_\eps \in \CC^2 \otimes \CC^{\eps} \otimes \CC^{\eps+1}$ by
\begin{align*}
L_\eps &\coloneqq b_1 \otimes \sum_{i=1}^\eps b_i \otimes b_i \,+\, b_2 \otimes \sum_{i=1}^\eps b_i \otimes b_{i+1}\\
&= b_1\otimes
\begin{psmallmatrix}
1 &   &  &   & 0  \\
  & 1 & & &  0\\[1pt]
  & & \ddots &  & \vdots \\[2pt]
  & & & 1 & 0
\end{psmallmatrix}
+
b_2\otimes
\begin{psmallmatrix}
0 &  1 &     &  &  \\
0 &     &  1   &  &  \\[1pt]
\vdots &     &  &\ddots &  \\[2pt]
0 &     &     &  & 1
\end{psmallmatrix}
\end{align*}
and for $\eta\in \NN$ define the tensor $N_\eta \in \CC^2 \otimes \CC^{\eta+1} \otimes \CC^\eta$ by
\begin{align*}
N_\eta &\coloneqq b_1 \otimes \sum_{i=1}^\eta b_i \otimes b_i \,+\, b_2 \otimes \sum_{i=1}^\eta b_{i+1} \otimes b_{i}\\
&= b_1\otimes
\begin{psmallmatrix}
1 &    &  & \\
   & 1 &  & \\
   &    & \ddots &  \\[2pt]
   &    & & 1 \\
0 & 0 & \cdots & 0
\end{psmallmatrix}
+
b_2\otimes
\begin{psmallmatrix}
0 & 0 & \cdots & 0\\
1 &     &  &  \\
   &  1   &  &  \\
   &  &\ddots &  \\[2pt]
   &  &  & 1
\end{psmallmatrix}.
\end{align*} 
\end{definition}

\begin{theorem}[Canonical form]\label{normalForm}
Let $t \in \CC^2\otimes \CC^n\otimes \CC^m$. There exist invertible linear maps $A\in \GL_2$, $B\in \GL_n$ and $C\in \GL_m$ and natural numbers $\eps_1,\ldots,\eps_{p},\eta_1,\ldots,\eta_q\in\NN$ and an $\ell\times \ell$~Jordan matrix $F$ such that, with $M=b_1\otimes I_\ell + b_2\otimes F$, we have
\begin{align}\label{canonicalform}
(A\otimes B\otimes C) t = \diag_{\CC^2}(0,L_{\eps_1},\ldots,L_{\eps_p},N_{\eta_1},\ldots,N_{\eta_q},M),
\end{align}
where the $0$ stands for some $0$-tensor of appropriate dimensions. 
The right-hand side of \eqref{canonicalform} is called the canonical form of $t$. 
\end{theorem}


Next we give a formula for the tensor rank of matrix pencils in canonical form (\cref{rankFormula}). \cref{rankFormula} is due to Grigoriev \cite{MR519843}, JáJá~\cite{MR539263} and Teichert~\cite{zbMATH04057703}, see also \cite[Theorem~19.4]{burgisser1997algebraic} or \cite[Theorem~3.11.1.1]{landsberg2012tensors}. 

\begin{definition}
Let $F$ be a Jordan matrix with eigenvalues $\lambda_1, \lambda_2, \ldots, \lambda_p$. Let $d(\lambda_i)$ be the number of Jordan blocks in $F$ of size at least two with eigenvalue~$\lambda_i$. Define $m(F) \coloneqq \max_i d(\lambda_i)$.
\end{definition}

\begin{theorem}\label{rankFormula}
Let $t=\diag_{\CC^2}(0,L_{\eps_1},\ldots,L_{\eps_p},N_{\eta_1},\ldots,N_{\eta_q}, b_1 \otimes I_\ell + b_2\otimes F)$ 
be a tensor in canonical form as in \eqref{canonicalform}. The tensor rank of~$t$ equals
\[
\rank(t)=\sum_{i=1}^p (\eps_i +1) +\sum_{i=1}^q (\eta_i+1)  + \ell + m(F).
\]
\end{theorem}

\begin{example}
Let $W_3 = b_2 \otimes b_1 \otimes b_1 +  b_1 \otimes b_2 \otimes b_1 +  b_1 \otimes b_1 \otimes b_2 \in (\CC^2)^{\otimes 3}$ as in \cref{multicounterexample}. The canonical form of $W_3$ is
\[
W_3 \cong b_1 \otimes \biggl(\begin{matrix} 1 & 0\\ 0 & 1\end{matrix}\biggr) + b_2 \otimes \biggl(\begin{matrix} 0 & 1\\ 0 & 0\end{matrix}\biggr).
\]
so in the notation of \cref{normalForm} we have $p = q = 0$ and $F = \begin{psmallmatrix} 0 & 1\\ 0 & 0\end{psmallmatrix}$. We can thus apply \cref{rankFormula} with $\ell = 2$ and $m(F) = 1$ to get $\rank(W_3) = 2+1 = 3$.
\end{example}

We are now ready to give the short proof of \cref{pencilprop}.

\begin{proof}[\upshape\bfseries Proof of \cref{pencilprop}]
Let $s \in \CC \otimes \CC^d \otimes \CC^d$, $t \in \CC^2 \otimes \CC^n \otimes \CC^m$. We may assume that $s = 1 \otimes \sum_{i=1}^r b_i \otimes b_i$ with $r = \rank(s)$.  By \cref{normalForm} we may assume that $t$ is in canonical form,
 $t=\diag_{\CC^2}(0, L_{\eps_1},\ldots,L_{\eps_p},N_{\eta_1},\ldots,N_{\eta_q},M)$. The tensor Kronecker product  $t\kron s$ is isomorphic to 
\[
t\kron \tpsi\cong \diag_{\CC^2}(\underbrace{t, \ldots, t}_r).
\]
By an appropriate local basis transformation we put this in canonical form
\[
t\kron \tpsi \cong \diag_{\CC^2}(L_{\eps_1}^{\oplus r},\ldots,L_{\eps_p}^{\oplus r},N_{\eta_1}^{\oplus r},\ldots,N_{\eta_q}^{\oplus r},M^{\oplus r}),
\]
which by \cref{rankFormula} has rank $r \cdot \rank(t) = \rank(s)\rank(t)$. 
\end{proof}
\begin{remark}
Proposition \ref{pencilprop} is also true over the finite field $\FF_q$ when $q\ge n,m$. To see this one may use the formula from \cite[Section 19.5]{burgisser1997algebraic} for the rank of pencils over finite fields, which for $q\geq n,m$ is as follows:
\[
\rank(t)=\sum_{i=1}^p (\eps_i +1) +\sum_{i=1}^q (\eta_i+1)  + \ell + \delta(B).
\]
Here $B$ is the \defin{regular part} of the pencil $t$ and $\delta(B)$ is the number of \defin{invariant divisors} of $B$ that do not decompose into a product of \defin{unassociated linear factors}. (We refer to \cite{burgisser1997algebraic} for definitions.) The invariant divisors of $\diag(B,\ldots,B)$ are just the invariant divisors of $B$ counted for each copy of $B$ and so Proposition \ref{pencilprop} follows.
\end{remark}

We note that part of the results in this section have been independently obtained in Section 2 of \cite{chen2017tensor}.

\paragraph{Acknowledgements}
We thank Jonathan Skowera for discussion, Fulvio Gesmundo for suggestions regarding Section~4, and Nick Vannieuwenhoven for discussion regarding the literature.
We acknowledge financial support from the European Research Council (ERC Grant Agreement no.~337603), the Danish Council for Independent Research (Sapere Aude), and VILLUM FONDEN via the QMATH Centre of Excellence (Grant no.~10059). JZ is supported by~NWO (617.023.116) and the QuSoft Research Center for Quantum Software.

\bibliographystyle{elsarticle-num}
\bibliography{all}

\end{document}